\documentclass[10pt]{amsart}
\usepackage{bm}
\usepackage{amssymb,amsfonts,amsmath,amsthm,graphicx}
\usepackage{comment}
\usepackage{slashbox,mathrsfs,mathtools}
\usepackage[all,poly]{xy}
\usepackage{float}

\allowdisplaybreaks

\begin{document}
\theoremstyle{plain}
\newtheorem{thm}{Theorem}[section]
\newtheorem{theorem}[thm]{Theorem}
\newtheorem*{theorem2}{Theorem}
\newtheorem{lemma}[thm]{Lemma}
\newtheorem{corollary}[thm]{Corollary}
\newtheorem{corollary*}[thm]{Corollary*}
\newtheorem{proposition}[thm]{Proposition}
\newtheorem{proposition*}[thm]{Proposition*}
\newtheorem{conjecture}[thm]{Conjecture}
\theoremstyle{definition}
\newtheorem{construction}[thm]{Construction}
\newtheorem{notations}[thm]{Notations}
\newtheorem{question}[thm]{Question}
\newtheorem{problem}[thm]{Problem}
\newtheorem{remark}[thm]{Remark}
\newtheorem{remarks}[thm]{Remarks}
\newtheorem{definition}[thm]{Definition}
\newtheorem{claim}[thm]{Claim}
\newtheorem{assumption}[thm]{Assumption}
\newtheorem{assumptions}[thm]{Assumptions}
\newtheorem{properties}[thm]{Properties}
\newtheorem{example}[thm]{Example}
\newtheorem{comments}[thm]{Comments}
\newtheorem{blank}[thm]{}
\newtheorem{observation}[thm]{Observation}
\newtheorem{defn-thm}[thm]{Definition-Theorem}



\title[Heat kernel coefficients on K\"ahler manifolds]{Heat kernel coefficients on K\"ahler manifolds}
\author{Kefeng Liu}
        \address{Center of Mathematical Sciences, Zhejiang University, Hangzhou, Zhejiang 310027, China;
                Department of Mathematics,University of California at Los Angeles,
                Los Angeles, CA 90095-1555, USA}
        \email{liu@math.ucla.edu}
\author{Hao Xu}
        \address{Center of Mathematical Sciences, Zhejiang University, Hangzhou, Zhejiang 310027, China;
        Department of Mathematics, University of Pittsburgh, 301 Thackeray Hall, Pittsburgh, PA 15260, USA}
        \email{mathxuhao@gmail.com}

        \begin{abstract}
       Polterovich proved a remarkable closed formula for heat kernel coefficients of the Laplace operator on compact Riemannian manifolds
       involving powers of Laplacians acting on the distance function. In the case of K\"ahler manifolds, we prove a combinatorial formula for powers of the complex Laplacian and use it to
        derive an explicit graph theoretic formula for the numerics in heat coefficients as
a linear combination of metric jets based on Polterovich's formula.
        \end{abstract}
    \maketitle

\section{Introduction}

The Laplace operator $\Delta$ on a Riemannian manifold $(M,g)$ of dimension $d$ is given by
$$\Delta=-\frac{1}{\sqrt{\det g}}\sum_{i,j=1}^d\partial_i(g^{ij}\sqrt{\det g}\,\partial_j).$$
The heat kernel is a smooth function $H(x, y, t)\in C^\infty(M\times M\times\mathbb R^+)$ that solves the
heat equation $\frac{\partial H}{\partial t}+\Delta_x H=0$ and satisfies $H(x,y,t)=H(y,x,t)$ and
$$\lim_{t\rightarrow 0}\int_M H(x,y,t)f(y)dV=f(x)$$
for any smooth function $f$ of compact support.

For example, the heat kernel of $\mathbb R^d$ is
$$H(x,y,t)=(4\pi t)^{-d/2}e^{-|x-y|^2/4t}.$$

If $M$ is compact, there is a unique heat kernel $H(x,y,t)$ on $M$ with the asymptotic expansion as $t\rightarrow 0^+$:
\begin{equation}\label{eqh}
H(x,x,t)=(4\pi t)^{-d/2}(a_0(x)+a_1(x)t+a_2(x)t^2+\cdots).
\end{equation}
It was first proved by Minakshisundaram-Pleijel \cite{MP}. Here the coefficients $a_j(x)$ are curvature invariants, i.e., invariant polynomials of jets of metrics.

For compact $M$, there exists a
complete orthonormal basis $\{\phi_0,\phi_1,\phi_2,\dots\}$ of $L^2(M)$, consisting of eigenfunctions of $\Delta$, with corresponding eigenvalues
$\lambda_0,\lambda_1,\lambda_2,\dots$ arranged in increasing order $0 = \lambda_0 < \lambda_1\leq\lambda_2\leq\cdots$, we have
$$H(x,y,t)=\sum_{k=0}^\infty e^{-\lambda_k t}\phi_k(x)\phi_k(y)$$
with uniformly convergence for any fixed $t$. Therefore by \eqref{eqh}, we get
$$\sum_{k=0}^\infty e^{-\lambda_k t}=\int_M H(x,x,t)=(4\pi t)^{-d/2}\left({\rm Vol}(M)+t \int_M \frac{1}{6}\rho\, dV+\cdots\right).$$

The heat kernel method was extensively used in index theory
\cite{Gil2,Ros,Zha}, moduli space \cite{Liu}, spectral geometry
\cite{Cha,Zel}, complex geometry \cite{MM} and quantum gravity
\cite{Kir}. Various recursive mechanisms of computing $a_k(x)$ were
developed, and explicit formulas were known for $k\leq 5$ (see e.g.,
\cite{Avr,Gil,Gri,Sak,vdV,Vas,Wei2}).

Polterovich \cite{Pol2} proved a remarkable closed formula for all
heat kernel coefficients using a generalization of the Agmon-Kannai
expansion \cite{AK}. Polterovich's formula was applied to obtain new
formulas for KdV hierarchy \cite{Pol,Pol2} and heat invariants of
spheres \cite{Pol3}.

\begin{theorem}[Polterovich \cite{Pol2,Pol3}] Let $w\geq 3n$. Then the heat kernel coefficients $a_n(x)$ are equal to
\begin{equation}\label{eqpol}
a_n(x)=(-1)^n\sum_{j=0}^{w}\binom{w+\frac{d}{2}}{j+\frac{d}{2}}\frac{1}{4^j j!(j+n)!}\Delta^{j+n}\big(f(dist(y,x)^{2})^j\big)|_{y=x},
\end{equation}
where $dist(y,x)$ is the distance function and $f$ is an arbitrary smooth function with $f(s)=s+O(s^2)$ for $s\in [0,\epsilon]$.
\end{theorem}
As noted in \cite{Pol2}, it is very difficult to convert powers of the Laplacian and the distance function to curvature tensors and
their covariant derivatives.

By Weyl's work on the invariants of the orthogonal group, any
curvature invariants on a Riemannian manifold can be formed from
Riemannian curvature tensor by covariant differentiations,
multiplications and contractions. In particular, the formal
expression of $a_n$  are universal curvature polynomials independent
of the dimension $d$.

\begin{remark}\label{rm2}
It was proved by Weingart \cite{Wei} that when $f(s)=s$,
Polterovich's formula holds for $w\geq n$. As noted by Weingart in a
private communication to us, when $f(s)=s$ Polterovich's formula was
originally proved under Riemannian normal coordinates, which are
quite different with K\"ahlerian case. In fact, Riemannian normal
coordinates are never complex analytic unless the manifold is flat,
which was first noticed by Bochner (cf. \cite[p. 22]{Ber}). In
Riemannian normal coordinates centered at the point $x = 0$, the
square distance is given locally by $dist(y,0)^2=|y|^2$. However in
K\"ahler normal coordinates centered at the point $z=0$, the square
distance contains higher order terms
$dist(z,0)^2=2|z|^2+\frac{2}{3}g_{i\bar jk\bar l}z_i\bar z_j z_k\bar
z_l+O(|z|^5)$. See Example \ref{h9}.
\end{remark}

Polterovich \cite[Thm. 2.1.3]{Pol3} showed that for general $f(s)$,
the proof of \eqref{eqpol} may be reduced to the case $f(s)=s$ under
a metric that needs only be Euclidean at the center. In fact,
repeating the proof of \cite[Thm. 1.2.1]{Pol2} gives the following
version of Polterovich's formula.
\begin{theorem}[Polterovich] Let $w\geq 3n$ and $(x_1,\dots,x_d)$ be local coordinates
on the Riemannian manifold $M$ such that the Riemannian metric at
the origin $x =0$ is Euclidean, i.e., $g_{ij}(0)=\delta_{ij}$. Then
\begin{equation}\label{eqpol2}
a_n(x)=(-1)^n\sum_{j=0}^{w}\binom{w+\frac{d}{2}}{j+\frac{d}{2}}\frac{1}{4^j
j!(j+n)!}\Delta^{j+n}(|x|^{2j})|_{x=0}.
\end{equation}
\end{theorem}

We could apply this version of Polterovich's formula in K\"ahler
normal coordinates, since the K\"ahler metric is Euclidean at the
center (up to a scaling). More precisely, if $(z_1,\dots,z_d)$ are
K\"ahler normal cooordinates centered at $z=0$. let
$z_i'=\sqrt{2}z_i$, $1\leq i\leq d$. Then $(Re z'_1,Im z'_1\dots,Re
z'_d,Im z'_d)$ is Euclidean at $z=0$.

Now we state the main result of this paper. We use graphs to
represent Weyl invariants in the sense that each vertex represents a
partial derivative of the K\"ahler metric and each edge represents
the contraction of index pairs, then the heat kernel coefficients
satisfy
\begin{equation}
a_n(z)=\sum_{G:\,w(G)=n}^{\text{stable}}z(G)\,G,
\end{equation}
where $G$ runs over stable digraphs of weight $n$ (cf. Definition \ref{df3})
and $z(G)$ is a graph invariant given by
\begin{equation}\label{eqcoeff}
z(G)=\frac{(-1)^{|V(G)|}2^{w(G)}}{|{\rm Aut}(G)|}\sum_{C\in\mathscr
C(G)}\frac{(-1)^{\mathfrak m(C)}\varphi(\Gamma_C)}{(\mathfrak
m(C)+w(G))!}.
\end{equation}
Here $\mathscr C(G)$ is the set of all $2^{|E(G)|}$ possible ways to cut edges of $G$ and $w(G)=|E(G)|-|V(G)|$ is the weight of $G$. Given an edge-cutting $C\in\mathscr C(G)$, $\mathfrak m(C)$
is the number of edges being cut and $\Gamma_C$ is a pointed graph obtained by connecting all loose ends to a new vertex $\bullet$ (cf. Definition \ref{df2}). Finally $\varphi(\Gamma_C)$ is
the number of stable reductions of $\Gamma_C$ (see Definition \ref{df1}).

The graph invariant $z(G)$ has the property that if $G$ is a disjoint
union of connected subgraphs $G=\cup_{i=1}^k G_i$, then we have
\begin{equation}
z(G)=\prod_{j=1}^k z(G_j)/|Sym(G_1,\dots,G_k)|,
\end{equation}
where $Sym(G_1,\dots,G_k)$ is the permutation group of the
connected subgraphs. It means that we only need to know $z(G)$ for connected graphs.

Our formula will be proved in Theorem \ref{main} of the next
section. In contrast to the graph theoretic formula for Bergman
kernel coefficients proved in \cite{Xu}, the invariant $z(G)$ here
may be nonzero for weakly connected graphs.

\begin{remark}
The combinatorial structure of the numerics in heat coefficients as
a linear combination of curvature tensors was considered intractable
(see e.g., \cite{Zel}). To the best of our knowledge, the formula
\eqref{eqcoeff} gives the first general result in this direction.
Although it is for K\"ahler manifolds, there are explicit
connections between K\"ahlerian and Riemannian curvature tensors.
Some special cases were treated in the appendix, which enables us to
use the formula \eqref{eqcoeff} to compute $a_1$ and $a_2$ for
Riemannian manifolds in \S \ref{secgraph}. Our work shows the
effectiveness of manipulating K\"ahler tensor in terms of graphs,
which shall also be useful in understanding other asymptotic
expansion coefficients in K\"ahler geometry. As asked by Zelditch,
it would be interesting to study whether the formula \eqref{eqcoeff}
may be used to derive inverse spectral results.

\end{remark}

\

\noindent{\bf Acknowledgements} We thank Peter Gilkey, Iosif
Polterovich, Steven Rosenberg, McKenzie Wang, Gregor Weingart and
Steven Zelditch for helpful comments.

\vskip 30pt
\section{A graph theoretic formula of heat kernel coefficients} \label{secheat}
We introduce some notations and concepts of graph theory.
\begin{definition}\label{df3}
In this paper, a graph always means a \emph{multi-digraph} $G=(V,E)$, which is defined to be a finite directed graph allowing to have multi-edges and
loops. Here $V$ and $E$ are the set of vertices and edges respectively. The \emph{weight} $w(G)$ of $G$ is defined to be
$|E|-|V|$. The
adjacency matrix $A=A(G)$ of a digraph $G$ with $n$ vertices is a
square matrix of order $n$ whose entry $A_{ij}$ is the number of
directed edges from vertex $i$ to vertex $j$. The {\it outdegree}
$\deg^+(v)$ and {\it indgree} $\deg^-(v)$ of a vertex $v$ are
defined to be the number of outward and inward edges at $v$
respectively.

A vertex $v$ of $G$ is called \emph{stable} if $\deg^-(v)\geq2,\ \deg^+(v)\geq2$.
We call $G$ \emph{stable} if each vertex $v$ is stable.

A vertex $v$ of $G$ is called \emph{semistable} if $\deg^-(v)\geq1, \deg^+(v)\geq1$ and $\deg^-(v)+\deg^+(v)\geq3$.
We call $G$ \emph{semistable} if each vertex $v$ is semistable.

A digraph $G$ is {\it strongly connected} if there is a
directed path from each vertex in $G$ to every other vertex.

\end{definition}

\begin{definition}
A \emph{pointed graph} is a multi-digraph $\Gamma=(V\cup\{\bullet\},E)$
with a distinguished vertex denoted by $\bullet$. Let $\Gamma_-$ be a subgraph of $\Gamma$ obtained by removing
the distinguished vertex $\bullet$.
We call $\Gamma$ \emph{semistable} (\emph{stable}) if each ordinary vertex $v\in V(\Gamma_-)$ is semistable
(stable). The \emph{weight} $w(\Gamma)$ of $\Gamma=(V\cup\{\bullet\},E)$ is defined to be $|E|-|V|$.
Denote
by ${\rm Aut}(\Gamma)$ the set of all automorphisms of
$\Gamma$ fixing the distinguished vertex $\bullet$.

\end{definition}

\begin{definition}
A directed edge $uv$ of a semistable pointed graph $\Gamma$ is called \emph{contractible} if $u\neq v$ and at least one of the following two conditions holds:
(i) $u\in V(\Gamma_-)$ and $\deg^+(u)=1$; (ii) $v\in V(\Gamma_-)$ and $\deg^-(v)=1$.

A semistable pointed graph $\Gamma$ is called \emph{stabilizable} if after contractions of a finite number of contractible edges of
$\Gamma$, the resulting graph becomes stable, which is called the \emph{stabilization graph} of $\Gamma$ and denoted by $\Gamma^s$.
\end{definition}

The following lemma was proved in \cite[Lem. 4.5]{Xu2}.
\begin{lemma}\label{h2}
Let $\Gamma$ be a semistable graph.
\begin{enumerate}
\item[(i)] If $\Gamma$ is strongly connected, then it is stabilizable.

\item[(ii)] If $\Gamma$ is stabilizable semistable graph and its stabilization graph $\Gamma^s$ is strongly connected, then $\Gamma$ is also strongly connected.
\end{enumerate}
\end{lemma}

\begin{definition}
Given a strongly connected pointed graph $\Gamma$, an edge $e$ is called \emph{redundant} if its removal from $\Gamma$ produces
a strongly connected subgraph. Otherwise, it is called \emph{essential}.
\end{definition}

\begin{lemma}\label{h1}
Every strongly connected semistable pointed graph with at least two vertices has a redundant edge.
\end{lemma}
\begin{proof}
Let $T^+(\bullet)$ and $T^-(\bullet)$ be directed spanning trees of $\Gamma$ rooted at $\bullet$ with all edges directed away from and towards $\bullet$ respectively. The existence of $T^+(\bullet)$ and $T^-(\bullet)$ is guaranteed by the strongly connectedness of $\Gamma$. Their union $T^+(\bullet)\cup T^-(\bullet)$ is a strongly connected spanning subgraph of $\Gamma$ and contains all essential edges of $\Gamma$.
This implies that $\Gamma$ contains at most $2|V(\Gamma)|-2$ essential edges.

If $\Gamma$ is stable and $|V(\Gamma)|\geq2$, then $2|E(\Gamma)|\geq 4(|V(\Gamma)|-1)+2$, i.e., $|E(\Gamma)|\geq 2|V(\Gamma)|-1$. Thus $\Gamma$ must contain
at least one redundant edge.

If $\Gamma$ is merely semistable, consider its stabilizatoin graph $\Gamma^s$, which must contain a redundant edge $e$. It is not difficult to see that
$e$ is also redundant in $\Gamma$.
\end{proof}

\begin{definition}\label{df1}
In relation to the expression of $\square^k$ in Proposition
\ref{h3}, we introduce a graph invariant $\varphi(\Gamma)$ for any
pointed graph $\Gamma$ as follows:

(i) if $\Gamma$ is not strongly connected, then $\varphi(\Gamma)=0$,

(ii) if $\Gamma$ has only one vertex with $l$ loops, then $\varphi(\Gamma)=l!$,

(iii) if there is an ordinary vertex $v\in V(\Gamma_-)$ that satisfies $\deg^+(v)=\deg^-(v)=1$, denote by $\Gamma/\{v\}$ the graph obtained by smoothing out $v$ in $\Gamma$ (i.e., removing $v$ and connecting its two neighboring vertices), then $\varphi(\Gamma)=\varphi(\Gamma/\{v\})$,

(iv) if $\Gamma$ is a strongly connected pointed graph, then
\begin{equation}
\varphi(\Gamma)=\sum_{e\in E(\Gamma)}\varphi(\Gamma-\{e\}),
\end{equation}
where $\Gamma-\{e\}$ denotes deleting an edge $e$ from $\Gamma$ while keeping the endpoints.

\end{definition}

\begin{table}[h] \footnotesize
\caption{$\varphi(\Gamma)$ for strongly connected $\Gamma$ with $w(\Gamma)\leq3$}\label{tb}
\begin{tabular}{|c|c|c|c|c|c|}
\hline $\bullet\, 1$
     & $\bullet\, 2$
     & $\xymatrix{
        *+[o][F-]{1} \ar@/^/[r]^1
         &
        \bullet\, \ar@/^/[l]^1}$
     & $\bullet\, 3$
     & $\xymatrix{
        *+[o][F-]{1} \ar@/^/[r]^1
         &
        \bullet\, 1 \ar@/^/[l]^1}$
     & $\xymatrix{ \circ \ar@/^/[r]^2 & \bullet \ar@/^/[l]^2 }$

\\
\hline $1$ & $2$ & $1$ & $6$ & $3$ & $8$
\\
\hline $\xymatrix{
        *+[o][F-]{1} \ar@/^/[r]^1
         &
        \bullet \ar@/^/[l]^2}$
     & $\xymatrix{
        *+[o][F-]{1} \ar@/^/[r]^2
         &
        \bullet \ar@/^/[l]^1}$
     & \begin{minipage}{0.6in}$\xymatrix@C=2mm@R=5mm{
                & \bullet \ar[dr]^{1}             \\
         \circ \ar[ur]^{1} \ar@/^0.3pc/[rr]^{1} & &  \circ  \ar@/^0.3pc/[ll]^{2}
         }$
         \end{minipage}
     & \begin{minipage}{0.6in}
             $\xymatrix@C=2mm@R=5mm{
                & \bullet \ar@/^-0.3pc/@<0.2ex>[dl]^{1}             \\
         \circ \ar@/^0.3pc/@<0.7ex>[ur]^{1} \ar@/^-0.3pc/@<0.2ex>[rr]^{1} & & *+[o][F-]{1}
         \ar@/^0.3pc/@<0.7ex>[ll]^{1}
         }$
         \end{minipage}
     & $\xymatrix{
        *+[o][F-]{2} \ar@/^/[r]^1
         &
        \bullet \ar@/^/[l]^1}$
     & \begin{minipage}{0.6in}$\xymatrix@C=2mm@R=5mm{
                & \bullet \ar[dr]^{1}             \\
        *+[o][F-]{1} \ar[ur]^{1}  & &  *+[o][F-]{1} \ar[ll]^{1}
         }$
         \end{minipage}
\\
\hline $4$ & $4$ & $4$ & $1$ & $2$ & $2$
\\
\hline
\end{tabular}
\end{table}

\begin{remark}\label{rm1}
Define a \emph{strong reduction} of a strongly connected pointed
graph $\Gamma$ to be a procedure: at each step removes a redundant
edge and smoothes out ordinary vertices $v\in V(\Gamma_-)$ with
$\deg^+(v)=\deg^-(v)=1$ until a single vertex $\bullet$ is reached.
Then $\varphi(\Gamma)$ counts the number of all strong reductions of
$\Gamma$. It is not difficult to see that a strong reduction of
$\Gamma$ removes exactly $w(\Gamma)$ edges, since smoothing out a
vertex reduces the number of edges by one. Lemma \ref{h1} implies
that $\varphi(\Gamma)>0$ when $\Gamma$ is strongly connected.
\end{remark}

\begin{lemma}\label{h8}
If two strongly connected semistable pointed graph $\Gamma_1$ and $\Gamma_2$
have the same stablization graphs, then $\varphi(\Gamma_1)=\varphi(\Gamma_2)$.
\end{lemma}
\begin{proof}
We only need to prove that for a strongly connected semistable pointed graph $\Gamma$
and its stabilization graph $\Gamma^s$, we have $\varphi(\Gamma)=\varphi(\Gamma^s)$.
This follows from the fact that if we remove a contractible edge from $\Gamma$, the
resulting graph is not strongly connected.
\end{proof}

 Let $(M,g)$ be a K\"ahler manifold of dimension
$d$.
We use Einstein summation convention that repeated indices are implicitly summed over. The indices
$i,j,k,\dots$ run from $1$ to $n$, while Greek indices
$\alpha,\beta,\gamma$ may represent either $i$ or $\bar i$. Let
$(g^{i\bar j})$ be the inverse of the matrix $(g_{i\bar j})$. We
also use the notation
$g_{i\bar
j\alpha_1\alpha_2\dots\alpha_m}:=\partial_{\alpha_1\alpha_2\dots\alpha_m}g_{i\bar j}$.

Recall that at each point $x$ on a K\"ahler manifold, there
exists a normal coordinate system such that at $x$ the K\"ahler metric satisfies
\begin{equation*}
g_{i\bar j}(x)=\delta_{ij}, \qquad g_{i\bar j k_1\dots
k_r}(x)=g_{i\bar j \bar l_1\dots\bar l_r}(x)=0
\end{equation*}
for all $r\leq N\in \mathbb N$, where $N$ can be chosen arbitrarily
large.

The curvature tensor is given by
\begin{equation}\label{eqcur1}
R_{i\bar jk\bar l} =-g_{i\bar j k\bar l}+g^{m\bar p}g_{m\bar j\bar
l}g_{i\bar p k}.
\end{equation}
The covariant derivative of a covariant tensor field
$T_{\beta_1\dots\beta_p}$ is defined by
\begin{equation}\label{eqcur4}
T_{\beta_1\dots\beta_p/\gamma}=\partial_{\gamma}T_{\beta_1\dots\beta_p}-\sum_{i=1}^p
\Gamma_{\gamma\beta_i}^{\delta}T_{\beta_1\dots\beta_{i-1}\delta\beta_{i+1}\dots\beta_p},
\end{equation}
where the Christoffel symbols $\Gamma_{\beta\gamma}^\alpha=0$ except
for $\Gamma_{jk}^i=g^{i\bar l}g_{j\bar l k}$, $\Gamma_{\bar j\bar
k}^{\bar i}=g^{l\bar i}g_{l\bar j\bar k}$. The above two identities
can be used to convert partial derivatives of metrics to covariant
derivatives of curvature tensors and vice versa around a normal
coordinate. If we use $D$ and $P$ to denote the conversions, we have
for example
\begin{gather*}
D(g_{i\bar j k\bar l})=-R_{i\bar j k\bar l},\qquad D(g_{i\bar j
k\bar l\alpha})=-R_{i\bar j k\bar l/\alpha},\\
P(R_{i\bar jk \bar l/ p\bar q})=-g_{i\bar jk \bar l p\bar q}
+g^{s\bar t}(g_{p\bar j s\bar l}g_{i\bar q k\bar t}+g_{k\bar j s\bar
l}g_{p\bar q i\bar t}+g_{i\bar j s\bar l}g_{k\bar q p\bar t}).
\end{gather*}
Closed-form expressions for $D$ and $P$ in terms of summations over
trees were obtained in \cite{XY}.

Thanks to the K\"ahler condition $\partial_i g_{j\bar k}=\partial_j g_{i\bar k}$ and $\partial_{\bar
l}g_{j\bar k}=\partial_{\bar k} g_{j\bar l}$, we can canonically associate a polynomial in the variables $\{g_{i\bar j\,\alpha}\}_{|\alpha|\geq1}$ to a stable graph $G$,
such that each vertex represents a partial derivative of $g_{i\bar j}$ and each edge represents the contraction of a pair of barred and unbarred indices.
Similarly a pointed graph $\Gamma$ represents a differential operator.

Next we prove a formula for $P(f_{/\beta_1\cdots\beta_k})$ as a
summation over stable pointed decorated trees (cf. {\cite[\S
4]{XY}}).
\begin{definition}
A \emph{pointed decorated tree} $T$ is a directed tree with a
distinguished vertex $\bullet$ such that each vertex of $T$ is
decorated by a finite number of outward and inward external-legs,
corresponding to unbarred and barred indices respectively. $T$ is
called \emph{stable} if each non-distinguished vertex is stable.
\end{definition}

Denote by $\mathscr P_f(\beta_1\cdots\beta_k)$ the set of pointed
decorated trees $T$ that can be obtained by starting from the
single-vertex $\bullet$ and letting $\beta_i,\,1\leq i\leq k$
consecutively act on vertices, external-legs or edges, corresponding
to (i), (ii) and (iii) in the following:
\begin{enumerate}

\item[i)] Action of a half-edge $i$ or $\bar i$ on a vertex $v$
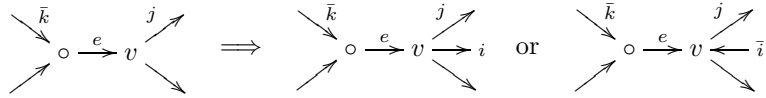
\begin{equation*}
\begin{tabular}{c}
\xymatrix@C=5mm@R=3mm{\ar[dr]^{\bar k} && &\\ &\circ \ar[r]^{e} & v
\ar[ur]^j \ar[dr] &
\\\ar[ur] &&&}
\end{tabular}
\Longrightarrow
\begin{tabular}{c}$\xymatrix@C=5mm@R=3mm{\ar[dr]^{\bar k} && &\\ &\circ
\ar[r]^{e} & v \ar[r]  \ar[ur]^j \ar[dr] & {\scriptstyle i}
\\\ar[ur] &&&}$ \end{tabular} \text{ or }
\begin{tabular}{c}$\xymatrix@C=5mm@R=3mm{\ar[dr]^{\bar k} && &\\
&\circ \ar[r]^{e} & v \ar[ur]^j \ar[dr] & {\scriptstyle \bar i}
\ar[l]
\\\ar[ur] &&&}$ \end{tabular}
\end{equation*}
\item[ii)] A half-edge $i$ or $\bar i$ may act on an external-leg with the same direction
 by putting them together on a new vertex
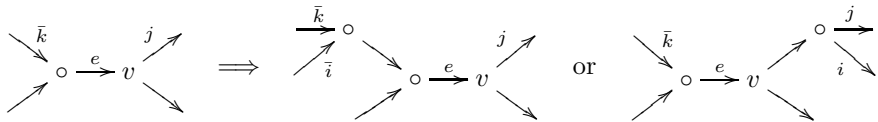
\begin{equation*}
\begin{tabular}{c}
\xymatrix@C=5mm@R=3mm{\ar[dr]^{\bar k} && &\\ &\circ \ar[r]^{e} & v
\ar[ur]^j \ar[dr] &
\\\ar[ur] &&&}
\end{tabular}
\Longrightarrow
\begin{tabular}{c}$\xymatrix@C=5mm@R=3mm{  \ar[r]^{\bar k} &\circ \ar[dr] & & &\\
\ar[ur]_{\bar i} & &\circ \ar[r]^e & v \ar[ur]^j \ar[dr]&
\\  & \ar[ur] & & &}$ \end{tabular}
\text{ or }
\begin{tabular}{c}$\xymatrix@C=5mm@R=3mm{ \ar[dr]^{\bar k} & & &\circ \ar[r]^j \ar[dr]_i &\\  &\circ \ar[r]^e & v \ar[ur] \ar[dr] & &
\\ \ar[ur] & & & &}$ \end{tabular}
\end{equation*}
\item[iii)] A half-edge $i$ or $\bar i$ may act on an edge $e$ by inserting it in the middle of $e$
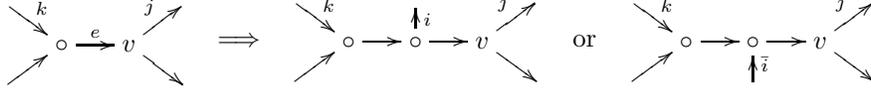
\begin{equation*}
\begin{tabular}{c}
\xymatrix@C=5mm@R=3mm{\ar[dr]^{\bar k} && &\\ &\circ \ar[r]^{e} & v
\ar[ur]^j \ar[dr] &
\\\ar[ur] &&&}
\end{tabular}
\Longrightarrow
\begin{tabular}{c}$\xymatrix@C=5mm@R=3mm{ \ar[dr]^{\bar k} & & & &\\  & \circ \ar[r] &\circ \ar[r] \ar[u]_i& v \ar[ur]^j \ar[dr]&
\\  \ar[ur]& & & &}$ \end{tabular}
\text{ or }
\begin{tabular}{c}$\xymatrix@C=5mm@R=3mm{ \ar[dr]^{\bar k} & & & &\\  & \circ \ar[r] &\circ \ar[r]& v \ar[ur]^j \ar[dr]&
\\ \ar[ur]& &\ar[u]_{\bar i} & &}$ \end{tabular}
\end{equation*}
\end{enumerate}

\begin{lemma} \label{tree} Let $k\geq0$. Then
\begin{equation}\label{eqtree}
P(f_{/\beta_1\cdots\beta_k})=\sum_{T} (-1)^{|V(T)|-1}\, T,
\end{equation}
where $T$ runs over $\mathscr P_f(\beta_1\cdots\beta_k)$.
\end{lemma}
\begin{proof}
By \eqref{eqcur4}, we have the recursive formula
$$P(f_{/\beta_1\dots\beta_{k-1}\beta_k})=\partial_{\beta_k}P(f_{/\beta_1\dots\beta_{k-1}})-\sum_{i=1}^{k-1}
\Gamma_{\beta_k\beta_i}^{\delta}P(f_{/\beta_1\dots\beta_{i-1}\delta\beta_{i+1}\dots\beta_{k-1}}).$$
Note that Case (ii) of the half-edge action corresponds to the
Christoffel symbols. Then one can prove \eqref{eqtree} inductively
by observing that no two trees in $\mathscr
P_f(\beta_1\cdots\beta_k)$ are identical, i.e., the procedure of
generating trees in $\mathscr P_f(\beta_1\cdots\beta_k)$ can be
reversed.
\end{proof}
\begin{example}
By Lemma \ref{tree},
\begin{align*}
&P(f_{/i j\bar k \bar l})\\=&
\begin{tabular}{c}
\xymatrix@C=7mm@R=1mm{ &&
\\ \ar[r]^{\bar k\bar l}&\bullet \ar[r] ^{ij} & \\
&&}
\end{tabular}
-
\begin{tabular}{c}
\xymatrix@C=5mm@R=3mm{ &  &
\\ \bullet \ar[r]  & \circ
\ar[ur]^{i j} &
\\  &&\ar[ul]^{\bar k\bar l}}
\end{tabular}
-
\begin{tabular}{c}
 \xymatrix@C=5mm@R=3mm{ &  &
\\ \bullet \ar[r]  & \circ
\ar[ur]^{ij} &
\\ \ar[u]_{\bar k}&&\ar[ul]^{\bar l}}
\end{tabular}
-
\begin{tabular}{c}
\xymatrix@C=5mm@R=3mm{ &  &
\\ \bullet \ar[r]  & \circ
\ar[ur]^{ij} &
\\  \ar[u]_{\bar l}&&\ar[ul]^{\bar k}}
\end{tabular}\\
=& f_{i j\bar k \bar l}-f_{p}g_{i\bar p j\bar k\bar l}-f_{p\bar
k}g_{i\bar p j\bar l}- f_{p\bar l}g_{i\bar k j\bar p}.
\end{align*}
Here $f_{i j\bar k \bar l}=\partial_{i j\bar k \bar l} f$ denotes
the partial derivative of $f$.
\end{example}

\begin{proposition}\label{h3} Let $\square$ be the complex Laplacian defined by $\square f=f_{/j\bar j}$ for any function $f$. Then
\begin{equation}\label{eqh7}
\square^k=\sum_{\Gamma:\,w(\Gamma)=k}^{\text{strong stable}}(-1)^{|V(\Gamma)|-1}\frac{\varphi(\Gamma)}{|{\rm Aut}(\Gamma)|}\, \Gamma,
\end{equation}
where $\Gamma$ runs over all strongly connected stable pointed graphs of weight $k$.
\end{proposition}
\begin{proof}
First we introduce the operation of attaching a directed arc $uv$ to a graph $\Gamma$.
We have two ways of attaching the endpoints $u,v$: (i) attach them to any vertices of $\Gamma$, (ii) attach them
to any edges of $\Gamma$ by creating a new vertex at the attaching edge.

Given $k\geq0$, denote by $S_k$ the multiset of graphs obtained by
attaching $k$ edges consecutively to $\bullet$ in all possible ways.
Obviously all graphs in $S_k$ are strongly connected. By
\eqref{eqcur1}, \eqref{eqcur4} and an similar argument of Lemma
\ref{tree}, we can prove
$$\square^k=\sum_{\Gamma\in S_k}(-1)^{|V(\Gamma)|-1}\,\Gamma.$$
Here the power of $-1$ is due to the fact that attaching an endpoint
to an edge will produce a negative sign, i.e.,
$\partial_{\alpha}g^{i\bar j}=-g^{p\bar j}g^{i\bar q}g_{p\bar
q\alpha}$.

Given a pointed graph $\Gamma$ of weight $k$, it is not difficult to see that there is a natural map from the set of all strong reductions of $\Gamma$ (cf. Remark \ref{rm1}) to $S_k$ by reversing the reduction procedure. Two strong reductions map to the same image in $S_k$ if and only if they differ by an automorphism of $\Gamma$, which implies that $\Gamma$ occurs in $S_k$ exactly $\frac{\varphi(\Gamma)}{|{\rm Aut}(\Gamma)|}$ times. So we conclude the proof.
\end{proof}

\begin{remark}
Eq. \eqref{eqh7} is an expression of $\square^k$ as a differential operator on functions.
Note that \eqref{eqh7} holds only at the center of a normal coordinate system. By Lemma \ref{h8}, if we allow semistable graphs in the summation on the right side of
\eqref{eqh7}, then it holds in the entire coordinate neighborhood.
See \cite[\S 4]{Xu2} for a detailed discussion.
\end{remark}

\begin{definition}\label{df3}
If a pointed graph $\Gamma$ satisfies $\deg^+(\bullet)=\deg^-(\bullet)$ and has $l$ loops at $\bullet$,
we define $\mathfrak m(\Gamma)=\deg^+(\bullet)-l$. When we remove $\bullet$ and all loops at $\bullet$ from $\Gamma$,
then there are $\mathfrak m(\Gamma)$ inward and $\mathfrak m(\Gamma)$ outward loose ends; an inward loose end may be paired with an outward loose end to form an edge. Denote by $\mathscr P(\Gamma)$ the set of all $\mathfrak m(\Gamma)!$ possible complete pairings of these loose ends. Given a pairing $P\in\mathscr P(\Gamma)$, denote by $G_P$ the graph generated from the pairing $P$. Note that two different complete parings may produce isomorphic graphs.
\end{definition}

\begin{definition} \label{df2}
Given a graph $G$, we can arbitrary cut $m$ edges of $G$ and get $2m$ loose ends, then we can construct a pointed graph by connecting all these loose ends to a new vertex $\bullet$. Denote by $\mathscr C(G)$ the set of all $2^{|E(G)|}$ possible edge-cuttings of $G$. Given an edge-cutting $C\in\mathscr C(G)$, denote by $\mathfrak m(C)$ the number of edges being cut and $\Gamma_C$ the corresponding pointed graph obtained by connecting all loose ends to a new vertex $\bullet$.
\end{definition}

\begin{lemma}\label{h4}
Let $\Gamma$ be a pointed graph that satisfies
$\deg^+(\bullet)=\deg^-(\bullet)$ and has $l$ loops at $\bullet$.
Denote $m=\deg^+(\bullet)-l$. As a differential operator, $\Gamma$
satisfies
\begin{equation}
\Gamma\left[\frac{|z|^{2(l+m)}}{(l+m)!}\right]_{z=0}=\frac{(l+m+d-1)!}{(m+d-1)!}\sum_{P\in\mathscr
P(\Gamma)}G_P,
\end{equation}
where $|z|=\sqrt{|z_1|+\cdots+|z_d|}$.
\end{lemma}
\begin{proof} By a straightforward computation, it is reduced to prove the following combinatorial identity:
for $m_1+\cdots+m_d=m$ with $m_i\geq0$,
$$\sum_{a_1+\cdots+a_d=l\atop a_i\geq0}\prod_{j=1}^d\binom{m_j+a_j}{m_j}=\binom{l+m+d-1}{m+d-1}.$$
The right-hand side has the following enumerative interpretation: in
a sequence of $l+m+d-1$ symbols, we choose $m+d-1$ objects and make
$d-1$ of them to be bars which are arranged as follows:
$$\underbrace{\star\dots\star}_{m_1}|\underbrace{\star\dots\star}_{m_2}|\cdots\cdots|\underbrace{\star\dots\star}_{m_d}$$
It is not difficult to see that all these arrangements are in
one-to-one correspondence with that of the left-hand side.
\end{proof}

\begin{lemma}\label{h5}
Let $\Gamma$ be a pointed graph that has no loops at the distinguished vertex $\bullet$. Then
\begin{equation}\label{eqh2}
\frac{1}{|{\rm Aut}(\Gamma)|}\sum_{P\in\mathscr P(\Gamma)} G_P=\sum_{G}\frac{\#\{C\in\mathscr C(G)\mid\Gamma_C\cong\Gamma\}}{|{\rm Aut}(G)|}\, G.
\end{equation}
\end{lemma}
\begin{proof}
The group ${\rm Aut}(\Gamma)$ has a natural action on the set $\mathscr P(\Gamma)$ by permuting the loose ends (and thus the pairings) resulting from removing $\bullet$. It is not difficult to see that the set of orbits corresponds to isomorphism classes of graphs generated from the parings, i.e., $\{G_P\mid P\in\mathscr P(\Gamma)\}$. Given a pairing $P\in\mathscr P(\Gamma)$ and $G=G_P$, denote by $E_P$ the set of edges in $G$ resulting from the pairing. Then the isotropy group at $P$ of the above action is ${\rm Aut}(G)'$, which is the subgroup of ${\rm Aut}(G)$ that leaves $E_P$ invariant. Note that ${\rm Aut}(G)'$ is determined up to conjugacy in ${\rm Aut}(G)$ for any $P\in\mathscr P(\Gamma)$ satisfying $G_P\cong G$.  Note also that the set $\{C\in\mathscr C(G)\mid\Gamma_C\cong\Gamma\}$ is in one-to-one correspondence with the coset of ${\rm Aut}(G)'$ in ${\rm Aut}(G)$. Therefore
$$\sum_{P\in\mathscr P(\Gamma)}\, G_P=\sum_{G}\frac{|{\rm Aut}(\Gamma)|}{|{\rm Aut}(G)'|}\, G=|{\rm Aut}(\Gamma)|\sum_{G}\frac{\#\{C\in\mathscr C(G)\mid\Gamma_C\cong\Gamma\}}{|{\rm Aut}(G)|}\, G,$$
as claimed.
\end{proof}
\begin{remark}
\eqref{eqh2} can be equivalently written as
\begin{equation}
\frac{\#\{P\in\mathscr P(\Gamma)\mid G_P\cong G\}}{|{\rm Aut}(\Gamma)|}=\frac{\#\{C\in\mathscr C(G)\mid\Gamma_C\cong\Gamma\}}{|{\rm Aut}(G)|}
\end{equation}
for any graph $G$ and pointed graph $\Gamma$.
\end{remark}

Before we prove our main result, we need a combinatorial lemma.
\begin{lemma}\label{h6} Let $w\geq m\geq1$ and $d\geq0$. Then
\begin{equation} \label{eqh1}
\sum_{j=m}^w (-1)^j\binom{w+d}{j+d}\binom{j+d-1}{m+d-1}=(-1)^m.
\end{equation}
\end{lemma}
\begin{proof}
Denote the left-hand side by $f(m,w)$. Obviously $f(m,m)=(-1)^m$. Using
$$\binom{w+d+1}{j+d}=\binom{w+d}{j+d}+\binom{w+d}{j+d-1},$$
it is not difficult to get
\begin{align*}
f(m,w+1)&=f(m,w)+\frac{(w+d)!}{(m+d-1)!(w+1-m)!}\sum_{j=m}^{w+1}(-1)^j\binom{w+1-m}{w+1-j}\\
&=f(m,w).
\end{align*}
Therefore \eqref{eqh1} follows by induction.
\end{proof}

\begin{theorem}\label{main}
On a $d$-dimensional K\"ahler manifold $M$, the heat kernel
coefficients of the (real) Laplacian is given by
\begin{equation}\label{eqh3}
a_n(z)=\sum_{G:\,w(G)=n}^{\text{stable}}\frac{(-1)^{|V(G)|}2^n}{|{\rm
Aut}(G)|}\sum_{C\in\mathscr C(G)}\frac{(-1)^{\mathfrak
m(C)}\varphi(\Gamma_C)}{(\mathfrak m(C)+n)!} \,G,
\end{equation}
where $G$ runs over all stable graphs of weight $n$ and $C$ runs over all edge-cuttings of $G$ (cf. Definition \ref{df2}).
\end{theorem}
\begin{proof}
First note that the real Laplacian $\Delta$ and complex Laplacian
$\square$ are related by $\square=-2\Delta$. By applying Proposition
\ref{h3} and Lemma \ref{h4} to Polterovich's formula \eqref{eqpol2},
we make the following calculation. Note that we take $w\geq 2n$
below.
\begin{align*}
a_n(z)&=(-1)^n\sum_{j=0}^w\binom{w+d}{j+d}\frac{1}{4^j j!(j+n)!}(-2\square)^{j+n}\big[(2|z|^2)^{j}\big]_{z=0}\\
&=2^n\sum_{j=0}^w\binom{w+d}{j+d}\frac{(-1)^j}{(j+n)!}\square^{j+n}\left[\frac{|z|^{2j}}{j!}\right]_{z=0}\\
&=2^n\sum_{j=0}^w\binom{w+d}{j+d}\frac{(-1)^j}{(j+n)!}\sum_{\Gamma:\,w(\Gamma)=j+n}^{\text{strong stable}}\frac{(-1)^{|V(\Gamma)|-1}\varphi(\Gamma)}{{\rm Aut}(\Gamma)}\, \Gamma\left[\frac{|z|^{2j}}{j!}\right]_{z=0}\\
&=2^n\sum_{j=0}^w\binom{w+d}{j+d}\frac{(-1)^j}{(j+n)!}\\
&\qquad\qquad\times\sum_{\Gamma:\,w(\Gamma)=j+n\atop \deg^+(\bullet)=\deg^-(\bullet)=j}^{\text{strong stable}}\!\!\!\frac{(j+d-1)!}{(\mathfrak m(\Gamma)+d-1)!}\frac{(-1)^{|V(\Gamma)|-1}\varphi(\Gamma)}{{\rm Aut}(\Gamma)}\sum_{P\in\mathscr P(\Gamma)}G_P.
\end{align*}
If $\Gamma$ has $l$ loops at $\bullet$ and $w(\Gamma)=j+n$, denote by $\Gamma'$ the graph obtained from $\Gamma$ by removing all loops at $\bullet$. It is not difficult to see (cf. Definition \ref{df1}) that $\varphi(\Gamma)=\binom{j+n}{l}\varphi(\Gamma')$ and ${\rm Aut}(\Gamma)=l!\cdot{\rm Aut}(\Gamma')$. Together with Lemma \ref{h5} and Lemma \ref{h6}, we get
\begin{align*}
a_n(z)&=2^n\sum_{j=0}^w\binom{w+d}{j+d}(-1)^j\\
&\qquad\qquad\times\sum_{\Gamma:\,w(\Gamma)=j+n\atop \deg^+(\bullet)=\deg^-(\bullet)=j}^{\text{strong stable}}
\binom{j+d-1}{\mathfrak m(\Gamma)+d-1}\frac{(-1)^{|V(\Gamma)|-1}}{(\mathfrak m(\Gamma)+n)!}\frac{\varphi(\Gamma')}{{\rm Aut}(\Gamma')}\sum_{P\in\mathscr P(\Gamma)}G_P\\
&=2^n\sum_{G:\,w(G)=n}^{\text{stable}}\;\sum_{m=1}^w\sum_{j=m}^w (-1)^j\binom{w+d}{j+d}\binom{j+d-1}{m+d-1}\frac{(-1)^{|V(G)|}}{(m+n)!}\\
&\qquad\qquad\qquad\qquad\qquad\qquad\qquad\qquad\qquad\qquad\qquad\times\sum_{C\in\mathscr C(G)\atop \mathfrak m(C)=m}\frac{\varphi(\Gamma_C)}{|{\rm Aut}(G)|} \,G\\
&=\sum_{G:\,w(G)=n}^{\text{stable}}\frac{(-1)^{|V(G)|}2^n}{|{\rm Aut}(G)|}\sum_{m=1}^{|E(G)|}\frac{(-1)^m}{(m+n)!}\sum_{C\in\mathscr C(G)\atop \mathfrak m(C)=m}\varphi(\Gamma_C) \,G\\
&=\sum_{G:\,w(G)=n}^{\text{stable}}\frac{(-1)^{|V(G)|}2^n}{|{\rm Aut}(G)|}\sum_{C\in\mathscr C(G)}\frac{(-1)^{\mathfrak m(C)}\varphi(\Gamma_C)}{(\mathfrak m(C)+n)!} \,G.
\end{align*}
The second to last equation used the crucial fact that $w\geq 2n\geq |E(G)|$.
\end{proof}

Denote by $z(G)$ the coefficient of $G$ in \eqref{eqh3},
$$z(G)=\frac{(-1)^{|V(G)|}2^{w(G)}}{|{\rm Aut}(G)|}\sum_{C\in\mathscr C(G)}\frac{(-1)^{\mathfrak m(C)}\varphi(\Gamma_C)}{(\mathfrak m(C)+w(G))!}.$$

\begin{proposition}\label{main2}
If $G$ is a disjoint
union of connected subgraphs $G=\cup_{i=1}^k G_i$, then we have
\begin{equation}\label{eqh4}
z(G)=\prod_{j=1}^k z(G_j)/|Sym(G_1,\dots,G_k)|,
\end{equation}
where $Sym(G_1,\dots,G_k)$ is the permutation group of the
connected subgraphs.
\end{proposition}
\begin{proof}
For any two graphs $G_1,G_2$, we have $\mathscr C(G_1\cup G_2)=\mathscr C(G_1)\times \mathscr C(G_2)$. Let
$$\tilde z(G)=\sum_{C\in\mathscr C(G)}\frac{(-1)^{\mathfrak m(C)}\varphi(\Gamma_C)}{(\mathfrak m(C)+w(G))!}.$$
If $C_1\in \mathscr C(G_1)$
and $C_2\in \mathscr C(G_1)$, then
$$\varphi(\Gamma_{C_1\times C_2})=\binom{\mathfrak m(C_1)+\mathfrak m(C_2)+w(G_1)+w(G_2)}{\mathfrak m(C_1)+w(G_1)}.$$
Now it is easy to see that $\tilde z(G_1\cup G_2)=\tilde z(G_1)\tilde z(G_2)$, which implies \eqref{eqh4}.
\end{proof}

\vskip 30pt
\section{Computations of $a_1$ and $a_2$} \label{secgraph}

On a K\"ahler manifold, we denote by $\mathcal R_{ijkl},\mathcal
Ric_{ij},\mathcal P$ the Riemannian curvature tensor, Ricci
curvature tensor and scalar curvature respectively, in the sense of
Riemmanian geometry, and denote by $R_{i\bar j k\bar l},Ric_{i\bar
j},\rho$ the corresponding K\"ahlerian curvature tensors. The
following lemma is well-known to experts. For completeness, we give
a proof in the appendix.
\begin{lemma}\label{h7}
On a K\"ahler manifold, we have the identities $|\mathcal R|^2=4|R|^2$, $|\mathcal Ric|^2=2|Ric|^2$ and $\mathcal P=2\rho$.
\end{lemma}

\begin{example}\label{h9}
We know that on a K\"ahler manifold, the square distance function
$\xi(z,z'):=dist(z,z')^2$ satisfies the following equation
\cite{Rie},
$$g^{i\bar j}(z)\frac{\partial \xi(z,z')}{\partial z_i}\frac{\partial \xi(z,z')}{\partial\bar z_j}=2\xi(z,z'),$$
which uniquely determines $\xi(z,z')$.

In particular, in K\"ahler normal coordinates centered at $z'=0$, we
have
$$g^{i\bar j}(z)=\delta_{ij}-g_{j\bar i k\bar l}z_k\bar z_l+O(|z|^3),$$
which implies
$$\xi(z,0)=dist(z,0)^2=2|z|^2+\frac{2}{3}g_{i\bar j k\bar l}z_i\bar z_j z_k\bar z_l+O(|z|^4).$$

By Remark \ref{rm2}, we could take $f(s)=s$ and $w=n$ in
Polterovich's formula \eqref{eqpol}. In particular,
$$a_1=-\frac{1}{2}\square^2(dist(z,0)^2).$$
Here we used the relation of real and complex Laplacians
$\Delta=-2\square$.

Since
$$\square^2=\left[\bullet\, 2\right]-\left[\xymatrix{
        *+[o][F-]{1} \ar@/^/[r]^1
         &
        \bullet\, \ar@/^/[l]^1}\right],$$
we have
\begin{align*}
a_1&=-\frac{1}{2}\left( \left[\bullet\,
2\right]\left(\frac{2}{3}g_{i\bar j k\bar l}z_i\bar z_j z_k\bar
z_l\right)- \left[\xymatrix{
        *+[o][F-]{1} \ar@/^/[r]^1
         &
        \bullet\, \ar@/^/[l]^1}\right]\big(2|z|^2\big)  \right)\\
        &=-\frac{1}{2}\left(\frac{8}{3}\left[\xymatrix{*+[o][F-]{2}}\right]-2\left[\xymatrix{*+[o][F-]{2}}\right]\right)\\
        &=-\frac{1}{3}\left[\xymatrix{*+[o][F-]{2}}\right]=-\frac{1}{3}g_{i\bar
i j\bar j}=\frac{1}{3}R_{i\bar i j\bar j}=\frac{1}{3}\rho,
\end{align*}
where $(i,\bar i),\,(j, \bar j)$ are paired indices to be
contracted.

By Lemma \ref{h7}, we get the well-known formula
$a_1=\frac{1}{6}\mathcal P$ in Riemannian tensors.
\end{example}
\smallskip

In the next example, we apply the graph theoretic formula
\eqref{eqh3} and Lemma \ref{h7} to compute $a_1$ and $a_2$ for
Riemmanian manifolds.

\begin{example} We represent a digraph as a weighted graph. The weight of a directed edge is
the number of multi-edges. The number attached to a vertex denotes the number of its self-loops. A
vertex without loops is denoted by a small circle $\circ$.

There is only one stable graph with weight $1$. By \eqref{eqh3},
\begin{equation}\label{eqa1}
a_1=-\frac{1}{3}\left[\xymatrix{*+[o][F-]{2}}\right],
\end{equation}
in agreement with the calculation in the previous example.

There are four stable graphs with weight $2$. By \eqref{eqh3},
\begin{align*}
a_2=&-\frac{2}{15}\left[\xymatrix{*+[o][F-]{3}}\right]+\frac{1}{18}\left[\xymatrix{*+[o][F-]{2}}\mid \xymatrix{*+[o][F-]{2}}\right]+\frac{23}{90}\left[\xymatrix{
        *+[o][F-]{1} \ar@/^/[r]^1
         &
        *+[o][F-]{1} \ar@/^/[l]^1} \right ]+ \frac{7}{45}
         \left[\xymatrix{ \circ  \ar@/^/[r]^2 & \circ \ar@/^/[l]^2
         }\right]\\
        =& -\frac{2}{15}g_{i\bar i j\bar j k\bar k}+\frac{1}{18}g_{i\bar i j\bar j}g_{k\bar k l\bar l} +\frac{23}{90} g_{i\bar i
k\bar l}g_{j\bar j l\bar k}+\frac{7}{45}g_{i\bar j k\bar l}g_{j\bar i
l\bar k}\\
=&-\frac{2}{15}\big(-\square \rho+|R|^2+2|Ric|^2\big)+\frac{1}{18}\rho^2+\frac{23}{90}|Ric|^2+\frac{7}{45}|R|^2\\
=&\frac{2}{15}\square \rho+\frac{1}{18}\rho^2-\frac{1}{90}|Ric|^2+\frac{1}{45}|R|^2.
\end{align*}

By Lemma \ref{h7} and $\Delta=-2\square$, we get the well-known
formulas of $a_1,a_2$ in Riemannian tensors.
$$a_1=\frac{1}{6}\mathcal P,\qquad a_2=-\frac{1}{30}\Delta\mathcal P+\frac{1}{72}\mathcal P^2-\frac{1}{180}|\mathcal Ric|^2+\frac{1}{180}|\mathcal R|^2.$$
\end{example}

\vskip 30pt

\appendix

\section{Riemannian and K\"ahler curvature tensors} \label{ap}
We give a proof of Lemma \ref{h7}.
Denote by $J$ the complex structure. If $e_1,\dots,e_{2d}$ is an orthonormal basis such that
$Je_i$ = $e_{d+i}$ for $i = 1,\dots,d$, then $y_i=\frac{1}{\sqrt{2}}(e_i-\sqrt{-1}Je_i)$, $1\leq i\leq d$ is a unitary basis.
By using $R(u,v,w,y)=R(Ju,Jv,w,y)=R(u,v,Jw,Jy)$, it is not difficult to get
\begin{align}\label{eqh6}
R(y_i,\bar y_j,y_k,\bar y_l)=& R(e_i,e_j,e_k,e_l)+R(e_i,Je_j,Je_k,e_l)\\
&\qquad  -\sqrt{-1}R(Je_i,e_j,e_k,e_l)-\sqrt{-1}R(e_i,e_j,Je_k,e_l).\nonumber
\end{align}
By definition, we have
\begin{align*}
|R|^2=&\sum_{i,j,k,l=1}^d R(y_i,\bar y_j,y_k,\bar y_l)R(y_j,\bar y_i,y_l,\bar y_k)\\
=&\sum_{i,j,k,l=1}^d R(y_i,\bar y_j,y_k,\bar y_l)\overline{R(y_i,\bar y_j,y_k,\bar y_l)}\\
=&\sum_{i,j,k,l=1}^d [R(e_i,e_j,e_k,e_l)+R(e_i,Je_j,Je_k,e_l)]^2\\
&\qquad +\sum_{i,j,k,l=1}^d[R(Je_i,e_j,e_k,e_l)+R(e_i,e_j,Je_k,e_l)]^2\\
=&\sum_{i,j,k,l=1}^d [R(e_i,e_j,e_k,e_l)^2+R(e_i,Je_j,Je_k,e_l)^2\\
&\qquad\qquad +R(Je_i,e_j,e_k,e_l)^2+R(e_i,e_j,Je_k,e_l)^2].
\end{align*}
The last equality follows from
\begin{align} \label{eqh5}
\sum_{i,j,k,l=1}^d R(e_i,e_j,e_k,e_l)R(e_i,Je_j,Je_k,e_l)&=0, \\
\sum_{i,j,k,l=1}^d R(Je_i,e_j,e_k,e_l)R(e_i,e_j,Je_k,e_l)&=0.
\end{align}
We prove the first identity, the proof of second identity is similar.
\begin{align*}
&\sum_{i,j,k,l=1}^d R(e_i,e_j,e_k,e_l)R(e_i,Je_j,Je_k,e_l)
\\
=&-\sum_{i,j,k,l=1}^d R(e_i,e_j,e_k,e_l)R(e_i,Je_j,e_k,Je_l)\\
=&-\sum_{i,j,k,l=1}^d R(e_i,e_j,e_l,e_k)R(e_i,Je_j,Je_l,e_k),
\end{align*}
which implies \eqref{eqh5}.
On the other hand,
\begin{align*}
|\mathcal R|^2=&\sum_{i,j,k,l=1}^d [R(e_i,e_j,e_k,e_l)^2+R(Je_i,Je_j,e_k,e_l)^2\\
&\qquad\qquad +R(e_i,e_j,Je_k,Je_l)^2+R(Je_i,Je_j,Je_k,Je_l)^2]\\
& +\sum_{i,j,k,l=1}^d [R(e_i,Je_j,Je_k,e_l)^2+R(e_i,Je_j,e_k,Je_l)^2\\
&\qquad\qquad +R(Je_i,e_j,Je_k,e_l)^2+R(Je_i,e_j,e_k,Je_l)^2]\\
& +\sum_{i,j,k,l=1}^d [R(Je_i,e_j,e_k,e_l)^2++R(e_i,Je_j,e_k,e_l)^2\\
&\qquad\qquad R(Je_i,e_j,Je_k,Je_l)^2+R(e_i,Je_j,Je_k,Je_l)^2]\\
& +\sum_{i,j,k,l=1}^d [R(e_i,e_j,Je_k,e_l)^2+R(e_i,e_j,e_k,Je_l)^2\\
&\qquad\qquad +R(Je_i,Je_j,Je_k,e_l)^2+R(Je_i,Je_j,e_k,Je_l)^2]\\
=&\; 4\sum_{i,j,k,l=1}^d [R(e_i,e_j,e_k,e_l)^2+R(e_i,Je_j,Je_k,e_l)^2\\
&\qquad\qquad +R(Je_i,e_j,e_k,e_l)^2+R(e_i,e_j,Je_k,e_l)^2]\\
=&\; 4|R|^2.
\end{align*}

Next we look at the Ricci curvature. By applying a unitary transformation if necessary, we may assume that both $Ric(y_i,\bar y_j)_{1\leq i,j\leq d}$ and $\mathcal Ric(e_i,e_j)_{1\leq i,j\leq 2d}$ are diagonal matrices. From \eqref{eqh6}, we have
\begin{align*}
Ric(y_i,\bar y_i)=&\sum_{k=1}^d R(y_i,\bar y_i,y_k,\bar y_k)\\
=&-\sum_{k=1}^d R(e_i,Je_i,e_k,Je_k)\\
=&\sum_{k=1}^d R(Je_i,e_k,e_i,Je_k)+\sum_{k=1}^d R(e_k,e_i,Je_i,Je_k)\\
=&\sum_{k=1}^d R(e_i,Je_k,Je_k,e_i)+\sum_{k=1}^d R(e_i,e_k,e_k,e_i)\\
=&\mathcal Ric(e_i,e_i)
\end{align*}
where we used the first Bianchi identity in the third equality. Similarly we can prove $Ric(y_i,\bar y_i)=\mathcal Ric(Je_i,Je_i)$. Therefore
$|\mathcal Ric|^2=2|Ric|^2$ and $\mathcal P=2\rho$.

$$ \ \ \ \ $$

\

\end{document}